\documentclass[12pt,a4paper]{amsart}
\usepackage{geometry}
\usepackage{amsmath,amssymb,amsfonts,amscd,amsthm,wasysym,color,enumitem,indentfirst,graphicx,booktabs,mathrsfs}
\usepackage[bookmarksnumbered,colorlinks,linktocpage,plainpages]{hyperref}
\hypersetup{pdfstartview={FitH}}
\usepackage{mathtools}
\usepackage{extarrows}

\numberwithin{equation}{section}
\newtheorem{theorem}{Theorem}[section]

\newtheorem{lemma}[theorem]{Lemma}
\newtheorem{proposition}[theorem]{Proposition}

\theoremstyle{definition}

\newtheorem{remark}[theorem]{Remark}

\allowdisplaybreaks[4]

\makeatletter
\@namedef{subjclassname@2020}{\textup{2020} Mathematics Subject Classification}
\makeatother

\newcommand{\R}{\mathbb R}
\newcommand{\Om}{\Omega}
\newcommand{\forms}{\mathcal{A}}
\newcommand{\cR}{\mathcal R}
\newcommand{\C}{\mathbb C}
\newcommand{\HH}{\mathbb H}
\newcommand{\Oh}{\mathbb O}
\newcommand{\Cl}{C\ell_n}
\newcommand{\A}{\mathcal A}
\newcommand{\Harm}{\mathcal H}

\newcommand{\Sc}{\operatorname{Sc}}

\newcommand{\Div}{\operatorname{div}}
\newcommand{\vol}{\mathrm{vol}}
\newcommand{\DO}{D_{\mathbb O}}
\newcommand{\DOb}{\overline D_{\mathbb O}}
\begin{document}
\title[Regular components and de Rham cohomology]
{Octonionic and Monogenic Completions of Harmonic Functions}
\date{October 21, 2026}

\author{Yong Li}
\address[Yong Li]{School of Mathematics and Statistics, Anhui Normal University, Wuhu 241002, Anhui, People's Republic of China}
\email{leeey@ahnu.edu.cn}

\author{Yuchen Zhang}
\address[Yuchen Zhang]{College of Mathematics and Systems Science, Shandong University of Science and Technology, Qingdao 266590, People's Republic of China}
\email{yuchen95@amss.ac.cn}

\thanks{
This work is supported by the National Natural Science Foundation of China (Grant No. 12501100), the Anhui Provincial Natural Science Foundation (Grant No. 2508085QA031).}

\subjclass[2020]{Primary 30G35 ; Secondary 31B05}
\keywords{Harmonic functions, Clifford module-valued monogenic functions, Octonionic regular functions, Vanishing of de Rham cohomology.  }

\begin{abstract}
Let $\Omega\subset\R^n$ be a domain, with $n\ge2$.
We characterize the real-valued harmonic functions that can
occur as prescribed real components of monogenic functions.
For any fixed nonzero real  Clifford module $M$
and any fixed unit vector $s\in M$, a smooth real-valued
function $h$ can be written as
$h=\langle F,s\rangle$ with $F$ monogenic if and only if
the $(n-1)$-form $*dh$ is exact.
Consequently, every real-valued harmonic function on
$\Omega$ admits such a representation if and only if
$H_{\mathrm{dR}}^{n-1}(\Omega)=0$.
In dimension eight, the same exactness criterion
characterizes the real parts of octonionic regular
functions and yields the corresponding equivalence with
$H_{\mathrm{dR}}^7(\Omega)=0$.
The proofs use global solvability of the Poisson equation
and elementary identities for differential forms.
\end{abstract}
\maketitle

\section{Introduction}

Let $\Omega\subset\R^n$ be a domain, where $n\ge2$, and let
$\Harm(\Omega)$ denote the space of real-valued harmonic
functions on $\Omega$.
In one complex variable, a classical theorem states that
a domain $\Omega\subset\C$ is simply connected if and only if
every $h\in\Harm(\Omega)$ is the real part of a holomorphic
function on $\Omega$.

Analogous questions arise in quaternionic analysis.
N\^ono~\cite{Non86} studied the existence of hyper-conjugate
harmonic functions in connection with domains of holomorphy.
That problem concerns prescribed complex components of
quaternionic functions and differs from the problem of
prescribing only the real part.
Sudbery~\cite{Sud79} showed that every real-valued harmonic
function is locally the real part of a Fueter-regular
function.
More recently, Li and Zhang~\cite{LZh26} obtained the
following global characterization: for a domain
$\Omega\subset\HH$, every real-valued harmonic function
on $\Omega$ is the real part of a Fueter-regular function
if and only if
\[
H_{\mathrm{dR}}^3(\Omega)=0.
\]

The components of octonionic regular functions and
Clifford-module-valued monogenic functions are also
harmonic; see, for example,
\cite{Li24,BDS82,DSS92,Gil91}.
It is therefore natural to ask which harmonic functions
can occur as their real parts or prescribed real
components, and which domains admit such representations
for every harmonic function.

Our main result answers both questions.

\begin{theorem}\label{thm:main}
Let $\Omega\subset\R^n$ be a domain.
\begin{enumerate}
\item
If $n=8$, then
\[
H_{\mathrm{dR}}^7(\Omega)=0
\]
if and only if every $h\in\Harm(\Omega)$ is the real part
of an octonionic regular function on $\Omega$.

\item
Let $n\ge2$.
Fix a nonzero real $\Cl$-module $M$ and a unit
vector $s\in M$.
Then
\[
H_{\mathrm{dR}}^{n-1}(\Omega)=0
\]
if and only if, for every $h\in\Harm(\Omega)$, there exists
an $M$-valued monogenic function $F$ on $\Omega$ such that
\[
h=\langle F,s\rangle.
\]
\end{enumerate}
\end{theorem}

The principal analytic input is Malgrange's global
solvability theorem for the Laplacian.
It yields the following differential-form characterization
of the relevant cohomology group.

\begin{proposition}\label{prop:cohomology}
Let $\Omega\subset\R^n$ be a domain, with $n\ge2$.
The following conditions are equivalent:
\begin{enumerate}
\item $H_{\mathrm{dR}}^{n-1}(\Omega)=0$;
\item $*dh$ is exact for every $h\in\Harm(\Omega)$.
\end{enumerate}
\end{proposition}

We then prove that, for an individual smooth real-valued
function $h$, exactness of $*dh$ is equivalent to the
existence of the corresponding regular or monogenic
completion.
This also shows that the obstruction is independent of
the choice of Clifford module and prescribed unit vector.

\section{Preliminaries}
\subsection{Differential forms and the Hodge Laplacian}

Equip $\Omega\subset\R^n$ with its Euclidean metric and
the orientation
\[
\vol=dx_1\wedge\cdots\wedge dx_n.
\]
Let $\A^k(\Omega)$ denote the space of smooth real-valued
$k$-forms on $\Omega$.
We use the scalar Laplacian
\[
\Delta=\sum_{j=1}^n\partial_j^2
\]
and write
\[
\Harm(\Omega)
=\{h\in C^\infty(\Omega,\R):\Delta h=0\}.
\]

Our conventions for the Hodge star \cite{Tu11} and the codifferential
are
\[
*^2\big|_{\A^k}=(-1)^{k(n-k)}I,
\qquad
\delta\big|_{\A^k}=(-1)^{n(k+1)+1}*d*.
\]
The Hodge Laplacian is
\[
\Box=d\delta+\delta d=-\Delta,
\]
where the last operator acts coefficientwise.
In particular,
\[
d \Box=\Box d,\qquad \delta \Box=\Box\delta.
\]
For a one-form $\alpha=\sum_j\alpha_j\,dx_j$, set
\[
\Div\alpha=\sum_j\partial_j\alpha_j=-\delta\alpha.
\]

%\subsection{Some differential geometry notations}
%Let $n\ge 2$, $\Om$ is a domain in $\R^n$, with coordinates $x_1,\ldots,x_n$, its Euclidean metric,
%and orientation $\mathrm{vol}=dx_1\wedge\cdots\wedge dx_n$.
% We write $\forms^k(\Om), k=0, 1, \cdots, n$ for the space of smooth and real-valued $k$-forms, and
%  $\Harm(\Om)=\{u\in C^\infty(\Om,\R):\Delta u=0\}$, where
%$\Delta=\sum_{j=0}^7\partial_j^2$ is the Laplace operator.
%
%We recall the following fundamental identities between differential operator $d$ and codifferential operator $\delta$.
%
%\[
% \delta\big|_{\forms^k}=(-1)^{n(k+1)+1}*d*,\qquad
% *^2\big|_{\forms^k}=(-1)^{k(n-k)},\qquad
% \Box=d\delta+\delta d=-\Delta,
%\]
%where the last operator acts coefficientwise. Here $*$ denotes the Hodge star operator, and $\Box$ denotes the Beltrami-Laplace operator on forms. In particular, for
%$\alpha=\sum_{j=1}^n\alpha_jdx_j$, its divengence
%\[
% \Div\alpha=\sum_{j=1}^n\partial_j\alpha_j=-\delta\alpha.
%\]

\subsection{Octonionic regular functions}
The  algebra of   octonions   $\mathbb{O}$  is  a non-commutative, non-associative,  normed 8 dimensional division algebra \cite{CS03} over  $\R$.  Let     $\mathbf{e_0}=1,\mathbf{e_1} \ldots,\mathbf{e_7}$ be its natural basis,  we have $$\mathbf{e}_i\mathbf{e}_j+\mathbf{e}_j\mathbf{e}_i=-2\delta_{ij},\quad i,j=1,\ldots,7.$$
 For $x=x_0+\sum\limits_{i=1}^7x_i\mathbf{e_i}\in \mathbb{O},\quad x_i\in\R$ we define octonion conjugation by $\overline{x}=x_0-\sum\limits_{i=1}^7x_i\mathbf{e_i} $, and real part operator $\operatorname{Re} x=x_0$, $\mathbb{O}$ is a 8-dimensional Euclidean space, under inner product $(p,q)=\operatorname{Re}(p\bar{q}), \quad p,q\in \mathbb{O}$.
The full multiplication table is conveniently encoded in  the Fano plane (see Figure 1 and \cite{Baz05}).
In the Fano plane, the vertices are labeled by  $\mathbf{e_1},\ldots, \mathbf{e_7}$.
Each of the 7 oriented lines gives a quaternionic triple. The
product of any two imaginary units is given by the third unit on the unique line
connecting them, with the sign determined by the relative orientation.

\begin{figure}[ht]
\centering
 \includegraphics[width=6cm]{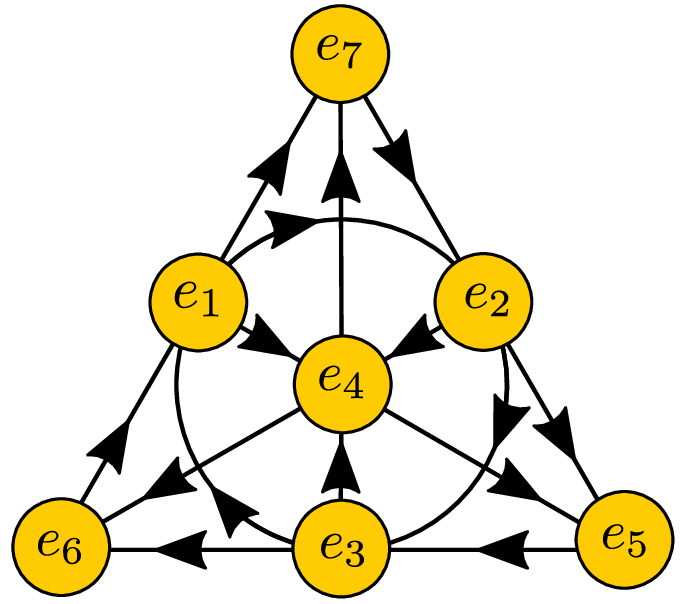}
  \caption{Fano plane}
  \label{fig:1}
\end{figure}
 Define the left differential operators
\[
 D_{\mathbb{O}}F=\sum_{j=0}^7 \mathbf{e}_j(\partial_j F),\qquad
 \overline {D_{\mathbb{O}}} F=\sum_{j=0}^7\overline {\mathbf{e}_j}(\partial_j F).
\]
A smooth function $F:\Om\to\Oh$ is \emph{octonionic regular} if $D_\Oh F=0$;
its space is denoted by $\cR(\Om)$. For completeness, the factorization
\begin{equation*}
\DO\DOb=\DOb\DO=\Delta I
\end{equation*}
holds despite the nonassociativity of $\Oh$.
In particular, every real component of an octonionic
regular function is harmonic.
We identify octonion-valued functions with one-forms by
\[
\varphi:C^\infty(\Omega,\Oh)\longrightarrow\A^1(\Omega),
\qquad
\varphi\left(\sum_{j=0}^7f_je_j\right)
=\sum_{j=0}^7f_j\,dx_j.
\]
For a real-valued function $h$,
\[
\phi(\DO h)=dh.
\]
%
%\begin{itemize}
%\item Let     $\mathbf{e_0}=1,\mathbf{e_1} \ldots,\mathbf{e_7}$ be its natural basis,  we have $$\mathbf{e_i}\mathbf{e_j}+\mathbf{e_j}\mathbf{e_i}=-2\delta_{ij},\quad i,j=1,\ldots,7.$$
%\item For $x=x_0+\sum_{i=1}^7x_i\mathbf{e_i}\in \mathbb{O},\quad x_i\in\R$ we define octonion conjugation by $\overline{x}=x_0-\sum_{i=1}^7x_i\mathbf{e_i} $, and real part operator $\operatorname{Re} x=x_0$, $\mathbb{O}$ is a 8-dimensional Euclidean space, under inner product $(p,q)=\operatorname{Re}(p\bar{q}), \quad p,q\in \mathbb{O}$.
%\item The full multiplication table is conveniently encoded in  the Fano plane (see Figure 1 and \cite{Baez}).
%In the Fano plane, the vertices are labeled by  $\mathbf{e_1},\ldots, \mathbf{e_7}$.
%Each of the 7 oriented lines gives a quaternionic triple. The
%product of any two imaginary units is given by the third unit on the unique line
%connecting them, with the sign determined by the relative orientation.
%
%\begin{figure}[ht]
%\centering
% \includegraphics[width=6cm]{Fano.jpg}
%  \caption{Fano plane}
%  \label{fig:1}
%\end{figure}
%
%
%
%\item Define the left differential operators
%\[
% DF=\sum_{j=0}^7 \mathbf{e_j}(\partial_j F),\qquad
% \overline D F=\sum_{j=0}^7\overline {\mathbf{e_j}}(\partial_j F).
%\]
%A smooth function $F:\Om\to\Oh$ is \emph{left regular} if $DF=0$;
%its space is denoted by $\cR(\Om)$.
%\end{itemize}

\subsection{Clifford module-valued monogenic functions}

For convenice, we  use the explicit construction of Clifford algebra $C\ell_n$ by the exterior algebra \cite[P.14 Second Construction]{Gil91}.
Set $$\Lambda^k(\R^n)^*=\operatorname{Span}_{\R}\{dx_{i_1}\wedge dx_{i_2}\wedge\cdots \wedge dx_{i_k}, 1\le i_1<i_2<\cdots<i_k\le n.\} $$
As a vector space,
\[
\Cl\cong\Lambda^\bullet(\R^n)^*
=\bigoplus_{k=0}^n\Lambda^k(\R^n)^*.
\]
 The basis unit $\mathbf{e}_i=dx_i, i=1,\cdots,n$, and under this identification left multiplication by $\mathbf{e}_i$
is
 $$\mathbf{e}_i\cdot=\wedge_{dx_i}-\iota_{\partial_i}, i=1,\cdots,n,$$
 where the $\wedge_{dx_i}$ is the exterior multiplication operator and $\iota_{\partial_i}$ is the interior multiplication
 \begin{align*}
   \wedge_{dx_i}(dx_{i_1}\wedge dx_{i_2}\wedge\cdots \wedge dx_{i_k})= & \ \ dx_i\wedge dx_{i_1}\wedge dx_{i_2}\wedge\cdots \wedge dx_{i_k}  \\
  \iota_{\partial_i}(dx_{i_1}\wedge dx_{i_2}\wedge\cdots \wedge dx_{i_k}) = & \ \ \sum_{j=1}^k(-1)^{j-1}\delta_{i{i_j}}dx_{i_1}\wedge\cdots \wedge \widehat{dx_{i_j}} \wedge \cdots \wedge dx_{i_k}.
 \end{align*}
 Thus a smooth $\Cl$-valued function is identified with an
inhomogeneous differential form.
For this regular module,
\begin{equation}\label{eq:clifford-dirac}
D_{\Cl}=\sum_{j=1}^n\mathbf{e}_j\partial_j=d+\delta.
\end{equation}

For a real Hilbert space $M$, $M$ is called a Clifford module\cite[P.60 Definition 7.44]{Gil91}, if there is a ring homomorphism $\rho:C\ell_n\to \operatorname{End}_{\R}(M)$ and $\rho(\mathbf{e}_i), i=1,\cdots, n$ is skew-adjoint, i.e, $\rho(\mathbf{e}_i)^*=-\rho(\mathbf{e}_i), i=1,\cdots, n$.

The associated Dirac operator is
 \[
D_MF=\sum_{j=1}^n\rho(\mathbf{e}_j)\partial_jF,
\qquad F\in C^\infty(\Omega,M).
\]
A smooth $M$-valued function is \emph{monogenic} if
$D_MF=0$.
The Clifford relations imply
\[
D_M^2=-\Delta I.
\]
Consequently, every continuous real-linear component of a
monogenic function is harmonic.
For general background, see
\cite{Aty64,BDS82,DSS92,Gil91,LM89,Har90}.

\section{Global Poisson solvability and exactness}
We use the following consequence of Malgrange's global
solvability theorem for elliptic operators;
see \cite[Chapter III]{Mal56} and \cite[P.59-P.60]{CSS04}.

\begin{theorem}[Malgrange]\label{thm:malgrange}
Let $\Om \subset\R^n$ be a domain, $n\geq2$. The operator
\[
 \Delta: C^\infty(\Om,\R)\longrightarrow C^\infty(\Om,\R)
\]
is surjective.
\end{theorem}

Applying Theorem~\ref{thm:malgrange} to each coefficient shows that
$\Box:\forms^k(\Om)\to\forms^k(\Om)$ is also surjective for every $k$.
This assertion imposes no boundary conditions.

\begin{lemma}\label{lem:closed}
Let $\Om\subset\R^n$ be a domain, $n\geq2$. Every \emph{closed}
$(n-1)$-form $\rho$ on $\Om$ can be written as
\[
 \rho=d\theta-*dh,
 \qquad \theta\in\forms^{n-2}(\Om),\quad h\in\Harm(\Om).
\]
\end{lemma}
\begin{proof}
Choose $\eta\in\forms^{n-1}(\Om)$ such that $\Box\eta=\rho$.
There is a real smooth function $h$ with $d\eta=h\,\mathrm{vol}$.
Since $d$ commutes with $\Box$ and $d\rho=0$,
\[
 -(\Delta h)\,\mathrm{vol}
 =\Box(h\,\mathrm{vol})
 =\Box(d\eta)=d(\Box\eta)=d\rho=0.
\]
Thus $h$ is harmonic. On top-degree forms our convention gives
$\delta(h\,\mathrm{vol})=-*dh$, so
\[
 \rho=\Box\eta=d(\delta\eta)+\delta(d\eta)
 =d(\delta\eta)-*dh.
\]
Take $\theta=\delta\eta$.
\end{proof}

\begin{proof}[Proof of Proposition~\ref{prop:cohomology}]
Assume (1). For harmonic $h$,
\[
 d(*dh)=(\Delta h)\,\mathrm{vol}=0.
\]
Thus $*dh$ is a closed $(n-1)$-form and hence is exact.

Conversely, assume (2) and let $\rho$ be any closed $(n-1)$-form.
By Lemma~\ref{lem:closed}, write $\rho=d\theta-*dh$ with $h$ harmonic.
Choose $\gamma\in\forms^{n-2}(\Om)$ such that $*dh=d\gamma$. Then
$\rho=d(\theta-\gamma)$. Every closed $(n-1)$-form is therefore exact,
which proves (1).
\end{proof}
The following lemma charaterized the condition $*dh$ is exact, which is prepared for the development of the subsequent text.

\begin{lemma}\label{lem:potential}
Let $\Om \subset\R^n$, $n\geq2$, be a domain and $\beta\in\forms^1(\Om)$.
The following are equivalent:
\begin{enumerate}
\item $*\beta$ is exact.
\item There exists $A\in\forms^2(\Om)$ such that $\beta=\delta A$.
\item There exists $\alpha\in\forms^1(\Om)$ such that $\beta=\delta d\alpha$.
\end{enumerate}
\end{lemma}
\begin{proof}
For any two-form $A$, the Hodge-star identities yield
\begin{equation}\label{eq:star-delta}
 *\delta A=d(*A).
\end{equation}
Indeed, the sign from $\delta A=(-1)^{3n+1}*d(*A)$ cancels the sign
$(-1)^{n-1}$ from applying $*$ a second time to the $(n-1)$-form $d(*A)$.

If $*\beta=d\gamma$, with $\gamma\in\forms^{n-2}(\Om)$, set $A=*\gamma$.
Since $**\gamma=\gamma$, equation~\eqref{eq:star-delta} gives
$*\delta A=d\gamma=*\beta$, proving (1)$\Rightarrow$(2).
Conversely, the same equation proves (2)$\Rightarrow$(1).

Suppose (2) holds. Choose $B\in\forms^2(U)$ with $\Box B=A$ and set
$\alpha=\delta B$. Using $\delta^2=0$ gives
\[
 \delta d\alpha=\delta d\delta B
 =\delta(d\delta+\delta d)B=\delta A=\beta.
\]
Thus (2)$\Rightarrow$(3), while (3)$\Rightarrow$(2) follows by taking $A=d\alpha$.
\end{proof}

\section{Real part of octonionic regular functions}
In view of Proposition \ref{prop:cohomology}, the first part of Theorem \ref{thm:main} can be deduced to the following theorem.
\begin{theorem}\label{thm:octonion-criterion}
Let $\Om\subset\R^8$ be a domain and $h\in C^\infty(\Om,\R)$.
There exists $F\in\cR(\Om)$ with $\operatorname{Re} F=h$ if and only if $*dh$ is exact.
In either case $h$ is harmonic.
\end{theorem}

We first introduce seven constant two-forms.
\[
\begin{aligned}
\omega_1&=dx_0\wedge dx_1-dx_2\wedge dx_3-dx_4\wedge dx_5+dx_6\wedge dx_7,\\
\omega_2&=dx_0\wedge dx_2+dx_1\wedge dx_3-dx_4\wedge dx_6-dx_5\wedge dx_7,\\
\omega_3&=dx_0\wedge dx_3-dx_1\wedge dx_2-dx_4\wedge dx_7+dx_5\wedge dx_6,\\
\omega_4&=dx_0\wedge dx_4+dx_1\wedge dx_5+dx_2\wedge dx_6+dx_3\wedge dx_7,\\
\omega_5&=dx_0\wedge dx_5-dx_1\wedge dx_4+dx_2\wedge dx_7-dx_3\wedge dx_6,\\
\omega_6&=dx_0\wedge dx_6-dx_1\wedge dx_7-dx_2\wedge dx_4+dx_3\wedge dx_5,\\
\omega_7&=dx_0\wedge dx_7+dx_1\wedge dx_6-dx_2\wedge dx_5-dx_3\wedge dx_4.
\end{aligned}
\]
Equivalently,
\begin{equation*}
\omega_j
=\frac12\sum_{i,k=0}^7
\langle \mathbf{e}_k,\mathbf{e}_i \mathbf{e}_j\rangle\,dx_i\wedge dx_k.
\end{equation*}
Actually, by the definition of $\omega_j, j=1,\cdots, 7$ , $\omega_j$ satisfies the following properties.
\begin{lemma}\label{lem:cayley}
For every real-valued smooth function $f$ and $1\leq j\leq7$,
\[
 \varphi^{-1}\bigl(\delta(f\omega_j)\bigr)=-(D_{\mathbb{O}}f)e_j.
\]
Hence, for $v=\sum\limits_{j=1}^7v_j\mathbf{e}_j, $ and
$W_v=\sum\limits_{j=1}^7v_j\omega_j$, one has
\[
 \varphi(D_{\mathbb{O}}v)=-\delta W_v,
\]
where $v_j\in C^\infty(\Om, \mathbb{R}),j=1,\cdots,7$.
\end{lemma}
\begin{proof}
Write $\omega_j=\frac12\sum\limits_{i,k=0}^7c^{(j)}_{ik}dx_i\wedge dx_k$, where
$c^{(j)}_{ik}=-c^{(j)}_{ki}$. The multiplication convention gives
\[
 c^{(j)}_{ik}=\langle \mathbf{e}_k,\mathbf{e}_i\mathbf{e}_j\rangle, \ \ \ \ i,k=0,\cdots 7, j=1,\cdots,7.
\]
For $i,k\geq1$ this also follows from the total antisymmetry of the
imaginary octonion structure constants; the cases with $i=0$ or $k=0$
follow from $\mathbf{e}_0=1$ and $\mathbf{e}_j^2=-1$.
Since
\[
 \delta(f\,dx_i\wedge dx_k)
 =\partial_kf\,dx_i-\partial_if\,dx_k,
\]
we obtain
\[
 \delta(f\omega_j)
 =-\sum_{i,k=0}^7\partial_if\,c^{(j)}_{ik}dx_k
 =-\varphi\left(\left(\sum_{i=0}^7\partial_if\,\mathbf{e}_i\right)\mathbf{e}_j\right).
\]
Summation over $j$ gives the second identity.
\end{proof}
Now we return to the proof of Theorem \ref{thm:octonion-criterion}.

\begin{proof}[Proof of Theorem \ref{thm:octonion-criterion}]
Suppose first that $F=h+v$ is regular, where
$v=\sum\limits_{j=1}^7v_j\mathbf{e}_j$ is purely imaginary with $v_j\in C^\infty(\Om, \mathbb{R}),j=1,\cdots,7$. By Lemma~\ref{lem:cayley},
\[
 0=\varphi(D_{\mathbb{O}}F)=dh-\delta W_v,
 \qquad W_v=\sum_{j=1}^7v_j\omega_j.
\]
Thus $dh=\delta W_v$, and equation~\eqref{eq:star-delta} gives
$*dh=d(*W_v)$.

Conversely, suppose $*dh$ is exact. By Lemma~\ref{lem:potential}, choose
$\alpha\in\forms^1(\Om)$ such that
\[
 \delta d\alpha=-dh.
\]
Set $a=\varphi^{-1}(\alpha)$, $s=\Div\alpha$, and define
\begin{equation*}
 v=-\overline{D_{\mathbb{O}}} a+s.
\end{equation*}
Here $s$ is regarded as a real octonion-valued function. Directly from the
basis multiplication rules,
\[
 \operatorname{Re}(\overline{D_{\mathbb{O}}} a)
 =\sum_{j=0}^7\partial_j\alpha_j=s.
\]
Consequently $\operatorname{Re} v=0$. And,
\[
 \begin{aligned}
 \varphi(D_{\mathbb{O}}v)
 &=-\varphi(\Delta a)+\varphi(D_{\mathbb{O}}s)\\
 &=\Box\alpha+ds\\
 &=(d\delta+\delta d)\alpha-d\delta\alpha
 =\delta d\alpha=-du.
 \end{aligned}
\]
It follows that $\varphi(D_{\mathbb{O}}(h+v))=dh-dh=0$. Thus $F=h+v$ is regular and
has real part $h$.

Finally, exactness of $*dh$ implies
\[
 0=d(*dh)=(\Delta h)\,\mathrm{vol},
\]
so $h$ is harmonic.
\end{proof}

\begin{proof}[Proof of Theorem~\ref{thm:main}(1)]
Apply Proposition~\ref{prop:cohomology} with $n=8$,
and then use Theorem~\ref{thm:octonion-criterion}.
\end{proof}

\section{Prescribed components of monogenic functions}
Now we turn to the proof of the second part of the Theorem\ref{thm:main}. We first consider the regular module $\Cl$.
Under its identification with the exterior algebra,
a Clifford-valued function has a well-defined scalar
component.
\begin{theorem}\label{thm:clifford-criterion}
Let $\Omega\subset\R^n$ be a domain, with $n\ge2$, and let
$h\in C^\infty(\Omega,\R)$.
There exists a $\Cl$-valued monogenic function $G$ with
$\Sc G=h$ if and only if $*dh$ is exact.
Either condition implies that $h$ is harmonic.
Moreover, whenever such a completion exists, one may
choose
\[
G\in\A^0(\Omega)\oplus\A^2(\Omega).
\]
\end{theorem}
\begin{proof}
Suppose that
\[
G=\sum_{k=0}^nG_k,\qquad
G_k\in\A^k(\Omega),\qquad G_0=h,
\]
is monogenic.
By \eqref{eq:clifford-dirac}, the degree-one component
of $D_{\Cl}G=0$ is
\[
dh+\delta G_2=0.
\]
Therefore
\[
*dh=-d(*G_2),
\]
so $*dh$ is exact.

Conversely, suppose that $*dh$ is exact.
By Lemma~\ref{lem:potential}, choose
$\alpha\in\A^1(\Omega)$ such that
\[
\delta d\alpha=dh.
\]
Set
\[
G=h-d\alpha.
\]
This function has only scalar and degree-two components,
and
\[
D_{\Cl}G
=(d+\delta)(h-d\alpha)
=dh-\delta d\alpha
=0.
\]
Thus $G$ is monogenic and $\Sc G=h$.

As before, exactness of $*dh$ implies $\Delta h=0$.
\end{proof}

Now we turn to the general Clifford module-valued monogenic functions cases. A general Clifford module has no distinguished scalar
component.
We therefore prescribe a real component by fixing a unit
vector.
\begin{theorem}\label{thm:module}
Let $\Omega\subset\R^n$ be a domain, with $n\ge2$.
Let $M$ be a nonzero real  $\Cl$-module and let
$s\in M$ be a fixed unit vector. For $h\in C^\infty(\Omega,\R)$, there exists a $M$-valued monogenic function $F$ with  $\langle F,s\rangle=h$ if and only if $*dh$ is exact.
\end{theorem}

\begin{proof}
For a Clifford module $M$, let  $\rho:C\ell_n\to\operatorname{End}_{\R}(M)$ be the canonical ring homomorphism.
Now suppose  $F\in C^\infty(\Om, M)$  is  monogenic with $\langle F,s\rangle=h$.
Define
\[
A_{ij}=\langle \rho(\mathbf{e}_i)\rho(\mathbf{e}_j)F,s\rangle\quad(i\ne j),
\qquad A_{ii}=0, \ \ \ i,j=1,\cdots, n.
\]
The Clifford relations give  $A_{ji}=-A_{ij}$, so
\[
A=\frac12\sum_{i,j}A_{ij}\,dx_i\wedge dx_j\in \forms^2(\Om).
\]
is a smooth 2-form.
For fixed $i\in\{1,\cdots,n\}$, according to  $D_{M}F=0$, we have
\[
0=\langle \rho(\mathbf{e}_i)D_MF,s\rangle
=\langle \rho(\mathbf{e}_i)\sum_{j=1}^n\rho(\mathbf{e}_j)\partial_jF,s\rangle=-\partial_i u+\sum_{j\ne i}\partial_jA_{ij}.
\]
On the other hand, we have that
\[
(\delta A)_i
=-\sum_j\partial_jA_{ji}
=\sum_j\partial_jA_{ij}.
\]
Thus $du=\delta A$, $*du$ is exact by Lemma \ref{lem:potential}.

Conversely, suppose that $*dh$ is exact.
By Theorem~\ref{thm:clifford-criterion}, there exists a $G\in C^\infty(\Om, C\ell_n)$ satisfies that
\[
G=h+Q,\qquad D_{C\ell_n}G=0,\qquad Q\in \forms^2(\Om).
\]
Set $F=\rho(G)s,$ then
\[
D_MF
=\sum_{j=1}^n\rho(\mathbf{e}_j)\partial_j\rho(G)s
=\rho(D_{C\ell_n}G)s=0.
\]
Noticed that when $i\ne j$
\[
(\rho(\mathbf{e}_i)\rho(\mathbf{e}_j))^*
=\rho(\mathbf{e}_j)^*\rho(\mathbf{e}_i)^*
=(-\rho(\mathbf{e}_j))(-\rho(\mathbf{e}_i))=-\rho(\mathbf{e}_i)\rho(\mathbf{e}_j).
\]
So $\rho(Q)=-\rho(Q)^*$ is skew-adjoint and
\[
\langle\rho(Q)s,s\rangle=0.
\]
Because $\rho$ is unital and $\|s\|=1$, we obtain
\[
\langle F,s\rangle
=h+\langle\rho(Q)s,s\rangle
=h.
\]
\end{proof}

\begin{proof}[Proof of Theorem~\ref{thm:main}(2)]
Combine Proposition~\ref{prop:cohomology}
with Theorem~\ref{thm:module}.
\end{proof}

\begin{remark}
  Here are two remarks before we end this paper.
  \begin{enumerate}
    \item The theorem also applies to any fixed nonzero continuous
real-linear functional $\ell:M\to\R$.
By the Riesz representation theorem,
\[
\ell(x)=\langle x,t\rangle
\]
for some nonzero $t\in M$.
Taking $s=t/\|t\|$ reduces the condition $\ell(F)=h$
to
\[
\langle F,s\rangle=h/\|t\|.
\]
Thus the same exactness criterion applies.
    \item  When $n=8$, there are two seemingly different descriptions of the condition
        $$H^7_{dR}(\Omega)=0.$$
        However, as shown in \cite{Li24}, $C\ell_8$ admits an explicit construction via the   octonion algebra, and under this construction, $\Oh\oplus \Oh$ is a $C\ell_8$-module.
         Thus octonionic regular functions can be regarded as a special case of monogenic functions taking values in a $C\ell_8$-module. The two descriptions are therefore intrinsically linked.
  \end{enumerate}
\end{remark}

\bigskip
\textbf{AI Statement} \ \
The authors used ChatGPT 6 Astra to improve the language and
presentation of the manuscript and to assist in developing
some of the mathematical arguments.
The authors take full responsibility for the content
of the manuscript.

\end{document}